\documentclass[psamsfont,a4paper,12pt]{amsart}
\usepackage[english]{babel}
\numberwithin{equation}{section}

\usepackage{setspace}
\usepackage{enumerate}
\usepackage{afterpage}
\usepackage[usenames,dvipsnames]{color}
\usepackage{graphicx}

\usepackage{amssymb,latexsym,amsmath}
\usepackage{amssymb}
\usepackage{mathrsfs}
\usepackage{amsfonts}
\usepackage{latexsym}
\usepackage[latin1]{inputenc}
\usepackage{pstricks,multido,pst-plot}
\usepackage[none]{hyphenat}

\usepackage{amsthm}
\usepackage{stackrel} 

\usepackage{tabularx}
\usepackage{longtable}

\usepackage{verbatim}
\usepackage{blindtext}
\usepackage{multicol,fancyheadings}
\usepackage{float}
\usepackage[vcentermath,enableskew]{youngtab}
\usepackage{subfig}

\allowdisplaybreaks

\newtheorem{theorem}{Theorem}[section]			

\theoremstyle{definition}

\newtheorem{example}{Example}[section]

\theoremstyle{remark}

\usepackage{tikz}
\usetikzlibrary{arrows}

\usepackage{paralist}
\usepackage{cite}

\title{Non-standard Zeckendorf decompositions; or, Tribonacci within Fibonacci}

\author[Anders]{Katie Anders}
\address{Department of Mathematics, University of Texas at Tyler, Tyler, TX 75799}
\email{kanders@uttyler.edu}

\author[Dawsey]{Madeline L. Dawsey}
\address{Department of Mathematics, University of Texas at Tyler, Tyler, TX 75799}
\email{mdawsey@uttyler.edu}

\author[Vandehey]{Joseph Vandehey}
\address{Department of Mathematics, University of Texas at Tyler, Tyler, TX 75799}
\email{jvandehey@uttyler.edu}


\begin{document}


\begin{abstract}
We study $B(n;k)$, the number of ways of writing $n$ as a sum or difference of the first $k$ Fibonacci numbers. We show that $B(0;k)$ satisfies the Tribonacci-like recurrence $B(0;k+1)=B(0;k)+B(0;k-1)+B(0;k-2)$ and that $B(n;k)$ satisfies a modified version of this recurrence.
\end{abstract}

\maketitle

\section{Introduction}\label{introduction}

The Fibonacci numbers $F_n$ satisfy the standard recurrence $F_{i}=F_{i-1}+F_{i-2}$ for $i\ge 3$, with initial values $F_1=1$ and $F_2=1$. The Zeckendorf decomposition of a positive integer $n$ is a representation of $n$ as a sum of non-consecutive Fibonacci numbers:
\[
n=F_{i_1}+F_{i_2}+\dots +F_{i_r},
\]
where $i_j\ge i_{j+1}+2$ for $j=1,2,\dots, r-1$, and $i_r\ge 1$.  This representation exists for all $n\ge 1$ and is unique \cite{zeckendorf1972representations}; however, if we relax the condition that $i_j\ge i_{j+1}+2$, then it is possible to have multiple representations of a given integer. 

One can define the function $R(N)$ to denote the number of ways to write $N$ as a sum of distinct Fibonacci numbers\footnote{For technical reasons, this means that $F_1$ and $F_2$ are considered the same and so $F_1+F_3$ and $F_2+F_3$ would be considered the same representation.}. Given the Zeckendorf decomposition, we see that $R(N)\ge 1$ for all $N\ge 1$, and $R(0)$ is said to be $1$ to account for the empty sum. The sequence of values of $R(N)$ is given by 
\[
1, 1, 1, 2, 1, 2, 2, 1, 3, 2, 2, 3, 1, 3, 3, 2, 4, 2, 3, 3, 1, 4, 3, 3, 5, 2, 4, 4, 2, 5, 3, 3, 4, 1, 4, 4, 3, 6,\dots
\]
and is A000119 in OEIS. This sequence and its various properties have been well-studied over the years \cite{bicknell1999number,carlitz1968fibonacci,chow2021fibonacci,ferns1965representation,klarner1966representations,stockmeyer2008smooth}. 

We also note that there is a connection between representations of an integer as a sum of Fibonacci numbers and representations as a sum of powers of $\phi=\frac{1+\sqrt{5}}{2}$ (see \cite{dekking2024counting}).

In this paper, we are interested in studying representations by sums and differences of distinctly indexed  Fibonacci numbers\footnote{So, unlike with $R(N)$, we do permit both $F_1$ and $F_2$ to appear in the same representation.}, that is, 
\[
n= \epsilon_1 F_{i_1}+ \epsilon_2F_{i_2}+\dots+\epsilon_r F_{i_r}, \qquad \epsilon_j\in \{-1,1\}\text{ for }1\le j\le r,
\]
with $i_j>i_{j+1}$ for $1\le j \le r-1$ and $i_r\ge 1$.  Results regarding such representations have not been as common in the literature. Alpert \cite{alpert2009differences} showed that each integer has a unique far-difference representation, defined as a sum and difference of Fibonacci numbers where every two terms of the same sign are at least 4 apart in index and every two terms of opposite sign are at least 3 apart in index. Bunder \cite{bunder1992zeckendorf} showed that each integer has a representation as a sum of non-consecutive negative-indexed Fibonacci numbers. Several other unique representations using sums and differences of Fibonacci numbers were studied by Hajnal \cite{hajnal2023short}.

Our goal is to count the number of ways to represent $n$ as a sum and difference of distinctly indexed Fibonacci numbers, similar to the authors' prior work on non-standard binary representations \cite{anders2024non} and non-standard quaternary representations \cite{anders2025non}. Unless we impose further restrictions, this count is always infinite, as 
\[
0=F_3-F_2-F_1=F_4-F_3-F_2=F_5-F_4-F_3=\cdots
\]
by the fundamental Fibonacci recurrence $F_i=F_{i-1}+F_{i-2}$. We therefore seek to count how many representations there are of a given length.

To be more precise with our terminology, we consider a digit set $\mathcal{D}=\{1,0,T\}$, where $T=-1$. The notation
\[
[a_{k} \ a_{k-1}\ \dots \ a_2\ a_1]_F, \qquad a_i \in \mathcal{D}\text{ for }1\le i \le k
\]
denotes the sum 
\[
\sum_{i=1}^{k} a_i F_i.
\]
We let $k$ refer to the length of this representation, even if $a_k=0$. Given a sequence $(b_1 \ b_2 \ \dots \ b_r)$ with $b_i\in \mathcal{D}$ for $1\le i \le r$, we say a representation $[a_k \ a_{k-1}\ \dots \ a_2\ a_1]_F$ starts with this sequence if $a_{k-i+1}=b_i$ for $1\le i \le r$. Note that this does require that $k \ge r$.

Let $B(n;k)$ denote the number of representations of $n$ of length $k$. For example, $B(1;3)=4$, since we can represent $1$ as $[0 \ 0 \ 1]_F=F_1$, $[0 \ 1 \ 0]_F=F_2$, $[1\ T \ 0 ]_F = F_3-F_2$, and $[1\ 0 \ T]_F=F_3-F_1$.

\begin{theorem}\label{thm:case 0}
    We have that $B(0;k)$ satisfies a Tribonacci-like recurrence relation:
    \[
    B(0;k+1)=B(0;k)+B(0;k-1)+B(0;k-2),\qquad k\ge 3,
    \]
    with $B(0;1)=1$, $B(0;2)=3$, and $B(0;3)=5$.
\end{theorem}

The sequence $B(0;k)$ begins with
\[
1, 3, 5, 9, 17, 31, 57, 105, 193, 355, 653, 1201, 2209, 4063, 7473, 13745, 25281, 46499.
\]
In OEIS, the sequence $B(0;k)$ is an offset version of A000213.

\begin{theorem}\label{thm:general case}
    In general, $B(n;k)$ satisfies a modified Tribonacci-like recurrence relation: 
    \[
    B(n;k+1)=B(n;k)+B(n;k-1)+B(n;k-2)+f(n;k+1),\qquad k\ge 3,
    \]
    where $f(n;k+1)$ is the number of representations of $n$ of length $k+1$ starting with $(1 \ 0)$, $(1 \ 1)$, $(T\ 0)$, $(T \ T)$, $(1 \ T \ 1)$, or $(T \ 1 \ T)$.

    Moreover, there exists a function $f(n)$ such that for any fixed $n$, $f(n;k)=f(n)$ for all sufficiently large $k$ (where the size may depend on the choice of $n$).
\end{theorem}

We include the beginning of some of these sequences below:
\begin{align*}
(B(1;k))_{k=1}^\infty &= (1,2,4,9,16, 30,56,\dots),\\
(B(2;k))_{k=1}^\infty &= (0,1,4,8, 16, 31, 57, \dots),\\
(B(3;k))_{k=1}^\infty &= (0,0, 2, 7, 15, 30, 57, \dots).
\end{align*}
None of these sequences appear in OEIS.

The sequence $f(n)$ is given by 
\[
(f(n))_{n=0}^\infty = (0,1,2,4,5,7, 9,11, 13, 15, 17, 20,21, 24, 26, 29, \dots),
\]
which also does not appear in OEIS.

This paper is organized as follows. In Section \ref{sec:reps of zero}, we begin by exploring representations of $n=0$. In Section \ref{sec:reps of nonzero}, we examine the general case of representations of $n$ and prove the first half of Theorem \ref{thm:general case}. When combined with the work of Section \ref{sec:reps of zero}, this provides a proof of Theorem \ref{thm:case 0}. Finally, in Section \ref{sec: exploring the f function}, we study the function $f(n;k)$ and prove the second half of Theorem \ref{thm:general case}.

In the proofs of these theorems, we rely on the known equality
\begin{equation}\label{eq:FibonacciSum}
\sum_{i=1}^mF_i=F_{m+2}-1.
\end{equation}
\section{Representations of Zero}\label{sec:reps of zero}

We begin with the case $n=0$, as it is simpler and instructive for the general case.

\begin{example}

    Here we explore the case $k=5$.  Consider a representation $[a_5 \ a_4 \ a_3 \ a_2 \ a_1]_F$ of $n=0$ counted by $B(0;5)$.  Suppose such a representation begins with $(1 \ 0)$.  Then this representation is minimized by letting $a_3=a_2=a_1=T$, which gives
    \[
    F_5+0\cdot F_4-F_3-F_2-F_1=5-2-1-1=5-4=1>0
    \]
    and thus is still too large to be a representation of $0$.
    Because a representation counted by $B(0;5)$ beginning with $(1 \ 1)$ yields a number even larger than one arising from a representation beginning with $(1 \ 0)$, we see that we also cannot represent $n=0$ with $a_5=1$ and $a_4=1$.  Thus any representation of $n=0$ beginning with $a_5=1$ must have $a_4=T$.

    Now suppose we have a representation of $n=0$ that is counted by $B(0;5)$ and starts with $(T \ 0)$ or $(T \ T)$.  Multiplying every coefficient in the representation by $T$ gives a representation of $n=0$ starting with $(1 \ 0)$ or with $(1 \ 1)$, which is a contradiction to our conclusion in the paragraph above.  Thus any representation of $n=0$ beginning with $a_5=T$ must have $a_4=1$.

    Lastly, suppose we have a representation of $n=0$ that is counted by $B(0;5)$ and begins with $(1 \ T \ 1)$.  Then the first part of the representation is $F_5-F_4+F_3=2F_3$.  Letting $a_2=a_1=T$, the latter part of the representation is $(-1)\sum\displaystyle_{i=1}^2 F_i=(-1)\left(F_4-1\right)=-F_4+1$.  Combining the first and latter parts, the entire representation of $n=0$ is $2F_3-F_4+1=F_3-F_2+1>0$.  Thus a representation of length $5$ for $n=0$ cannot begin with $(1 \ T \ 1)$.  By an argument similar to the contradiction in the paragraph above, a representation of length 5 for $n=0$ also cannot begin with $(T \ 1 \ T)$.

In summary, a representation counted by $B(0;5)$ cannot begin with $(1 \ 0)$, $(1 \ 1)$, $(T \ 0)$, $(T \ T)$, $(1 \ T \ 1)$, nor $(T \ 1 \ T)$.

\end{example}

Similar arguments hold for general $k\geq3$ when $n=0$.

\begin{theorem}\label{thm:case0openings}
Any representation of $n=0$ of length $k\ge 3$ cannot begin with $(1 \ 0)$, $(1 \ 1)$, $(T \ 0)$, $(T \ T)$, $(1 \ T \ 1)$, nor $(T \ 1 \ T)$.
\end{theorem}

\begin{proof}
    Suppose a representation of length $k$ starts with $(1 \ 0)$.  Then this representation is minimized by letting $a_{k-2}=a_{k-3}=\cdots=a_1=T$, which gives
    \[
    F_k+0\cdot F_{k-1}-\sum_{i=1}^{k-2} F_i=F_k-F_k+1=1>0,
    \]
    where we have used Equation \eqref{eq:FibonacciSum}. Thus this or any other representation starting with $(1\ 0)$ is too large to be a representation of $n=0$.
    Because a representation beginning with $(1 \ 1)$  yields a number even larger than one arising from a representation beginning with $(1 \ 0)$, we see that we also cannot represent $0$ starting with $(1 \ 1)$. 

    Now suppose a representation of length $k$ for $n=0$ starts with $(T \ 0)$ or with $(T \ T)$.  Multiplying every coefficient in this representation by $T$ gives a length $k$ representation of $n=0$ starting with $(1 \ 0)$ or with $(1 \ 1)$, the existence of which is a contradiction to our conclusion in the paragraph above. 

    Lastly, suppose we have a length $k$ representation beginning with $(1 \ T \ 1 )$. Then the first part of the representation is $F_k-F_{k-1}+F_{k-2}=2F_{k-2}$. Minimizing the rest of the representation by letting $a_{k-3}=a_{k-4}=\cdots=a_1=T$, we have
    \[
    F_k-F_{k-1}+F_{k-2}-\sum_{i=1}^{k-3} F_i = 2F_{k-2}-(F_{k-1}-1)=F_{k-2}-F_{k-3}+1>0,
    \]
    which means any representation starting with $(1 \ T \ 1)$ is too large to be a representation of $n=0$.  Thus a representation of $n=0$ cannot begin with $(1 \ T \ 1)$. By an argument similar to the contradiction in the paragraph above, a representation of $n=0$ also cannot begin with $(T \ 1 \ T)$. This completes the proof.
\end{proof}
It is an immediate consequence of the above to see that any representation of $0$ of length $k\ge 3$ must begin with $(1 \ T \ 0)$, $(1 \ T \ T)$, $(T \ 1 \ 0)$, $(T \ 1\ 1 )$, or $(0)$.

We can show by hand that $B(0;1)$ counts the singular representation $[0]_F$; $B(0;2)$ counts the three representations $[0\ 0]_F$, $[1 \ T]_F$, and $[T \ 1]_F$; and $B(0;3)$ counts the five representations $[0\ 0 \ 0]_F$, $[0 \ 1 \ T]_F$, $[0 \ T \ 1]_F$, $[1 \ T \ T]_F$, and $[T\ 1 \ 1]_F$.

When these facts are combined with Theorems \ref{thm:general case} and \ref{thm:case0openings}, this section provides a proof of Theorem \ref{thm:case 0}.

\section{Representations of a Nonzero Integer $n$}\label{sec:reps of nonzero}

In this section, we prove the first part of Theorem \ref{thm:general case}. In particular, we show that $B(n;k)$ satisfies a modified Tribonacci-like recurrence relation: 
    \[
    B(n;k+1)=B(n;k)+B(n;k-1)+B(n;k-2)+f(n;k+1),\qquad k\ge 3,
    \]
    where $f(n;k+1)$ is the number of representations of $n$ of length $k+1$ starting with $(1 \ 0)$, $(1 \ 1)$, $(T\ 0)$, $(T \ T)$, $(1 \ T \ 1)$, or $(T \ 1 \ T)$.

\begin{proof}
Let $n$ be a nonzero integer and $k+1$ be an integer greater than or equal to $4$.  Consider a representation of $n$ of length $k+1$ starting with $(0)$.  Such a representation is of the form $n=[0 \ a_{k} \ a_{k-1} \ \dots \ a_2 \ a_1]_F$ and corresponds to a representation of $n$ using only $k$ Fibonacci numbers, namely $n=[a_{k}\ a_{k-1} \ \dots \ a_2\ a_1]_F$. This is clearly a bijection between representations of $n$ of length $k+1$ starting with $(0)$ and representations of $n$ of length $k$, which are counted by $B(n;k)$. 

Next consider a representation of $n$ of length $k+1$ starting with $(1 \ T \ 0)$. Such a representation is of the form $n=[1\ T \ 0 \ a_{k-2} \ a_{k-3} \ \dots \ a_2 \ a_1]_F$ and, by the fundamental Fibonacci recurrence $F_{k+1}-F_{k}=F_{k-1}$, corresponds to a representation of $n$ using $k-1$ digits and starting with $(1)$, namely $n=[1\ a_{k-2} \ a_{k-3} \ \dots \ a_2 \ a_1]_F$. Again, it is clear this is a bijection between representations of $n$ of length $k+1$ starting with $(1 \ T \ 0)$ and representations of $n$ of length $k-1$ starting with $(1)$. 

By repeating the argument of the previous paragraph with the roles of $T$ and $1$ reversed, we see that representations of $n$ of length $k+1$ starting with $(T\ 1\ 0)$ are in correspondence with representations of $n$ of length $k-1$ starting with $(T)$.

Lastly, for now, consider a representation of $n$ of length $k+1$ beginning with $(1 \ T \ T)$, $n=[1 \ T \ T\ a_{k-2} \ a_{k-3} \ \dots \ a_2 \ a_1]_F$. Since $F_{k+1}-F_{k}-F_{k-1}=0$, this representation corresponds to a representation of $n$ using $k-1$ digits and starting with $(0)$, namely $n=[0\ a_{k-2} \ a_{k-3} \ \dots \ a_2 \ a_1]_F$. Thus representations of $n$ of length $k+1$ starting with $(1 \ T \ T)$ are in a bijective correspondence with representations of $n$ of length $k-1$ starting with $(0)$.

The previous three paragraphs show that representations of $n$ of length $k+1$ starting with $(1 \ T \ 0)$, $(T\ 1 \ 0)$, or $(1 \ T \ T)$ are in bijection with representations of $n$ of length $k-1$ starting with $(1)$, $(T)$, or $(0)$, respectively, but these three categories form a partition of the set of all representations of $n$ of length $k-1$. Thus, the number of representations of $n$ of length $k+1$ starting with $(1 \ T \ 0)$, $(T\ 1 \ 0)$, or $(1 \ T \ T)$ is equal to $B(n;k-1)$. 

Finally, consider a representation of $n$ of length $k+1$ beginning with $(T \ 1 \ 1)$.  Such a representation is of the form $n=[T \ 1 \ 1\ a_{k-2} \ a_{k-3} \ \dots \ a_2 \ a_1]_F$. Again applying the Fibonacci recurrence $- F_{k+1}+F_{k}+F_{k-1}=0$, there is a corresponding representation of $n$ using the remaining $k-2$ digits of this representation, namely $n=[a_{k-2} \ a_{k-3} \ \dots \ a_2 \ a_1]_F$. As before, this is a bijection, and so the number of representations of $n$ of length $k+1$ beginning with $(T \ 1 \ 1)$ is equal to $B(n;k-2)$. 

Thus the number of representations of $n$ of length $k+1$ starting with $(0)$, $(1 \ T \ 0)$, $(T \ 1 \ 0)$, $(1 \ T \ T)$, or $(T\ 1 \ 1)$  is equal to $B(n;k)+B(n;k-1)+B(n;k-2)$.
Therefore, we have that 
    \[
    B(n;k+1)=B(n;k)+B(n;k-1)+B(n;k-2)+f(n;k+1),\ k\ge 3,
    \]
    where $f(n;k+1)$ counts all representations of $n$ of length $k+1$ that do \emph{not} begin with $(0)$, $(1 \ T \ 0)$, $(T \ 1 \ 0)$, $(1 \ T \ T)$, or $(T\ 1 \ 1)$. It is simple to check that this corresponds to the function $f(n;k+1)$ as it was written in the statement of the theorem.
\end{proof}

\section{The Mysterious Sequence $f(n;k)$}\label{sec: exploring the f function}

In this section, we will prove the second half of Theorem \ref{thm:general case}, that for any fixed $n$, there is a function $f(n)$ such that $f(n;k)=f(n)$ for all sufficiently large $k$. Recall that $f(n;k+1)$ is the number of representations of $n$ of length $k+1$ starting with $(1 \ 0)$, $(1 \ 1)$, $(T\ 0)$, $(T \ T)$, $(1 \ T \ 1)$, or $(T \ 1 \ T)$.

If a representation of $n$ using $k+1$ digits starts with $(1 \ 1)$, then even if $a_i=T$ for all other $i$, we still get
\[
F_{k+1}+F_{k}-\sum_{i=1}^{k-1}F_i=F_{k+1}+F_{k}-(F_{k+1}-1)=F_{k}+1.
\]
This is larger than $n$ once $k+1$ is sufficiently large, so $f(n;k+1)$ will not count any such representation once $k+1$ is large enough.

If a representation of $n$ of length $k+1$ starts with $( 1 \ 0 \ 0)$, $(1 \ 0 \ 1)$, or $(1 \ T \ 1)$, then even if all other digits $a_i=T$, we still get at least
\begin{align*}
F_{k+1}-F_k+F_{k-1}-\sum_{i=1}^{k-2}F_i&=F_{k+1}-F_k+F_{k-1}-(F_{k}-1)=(F_{k+1}-F_{k})- (F_k-F_{k-1})+1\\
&=F_{k-1}-F_{k-2}+1=F_{k-3}+1.
\end{align*}
Again, this is larger than $n$ for $k+1$ sufficiently large, so $f(n;k+1)$ will not count any such representation once $k+1$ is large enough. 

By applying the arguments of the previous two paragraphs after multiplying everything by $T$, we see that $f(n;k+1)$ will not count any representation starting with $(T \ T)$, $(T \ 0  \ 0)$, $(T\ 0 \ T)$, nor $(T \ 1 \ T)$ once $k+1$ is large enough.

If $f(n;k+1)$ is the sequence defined in the introduction, then $f(n;k+1)$ is by definition the number of representations of $n$ using the first $k+1$ Fibonacci numbers, where the representation must start with $(1 \ 0), (1 \ 1), (T \ 0), (T \ T), (1 \ T \ 1)$, or $(T \ 1 \ T)$. Thus, for $k+1$ large enough, $f(n;k+1)$ will only be counting representations starting with $(1 \ 0 \ T)$ or with $(T \ 0\ 1)$ by the arguments above. Define $K$ so that this ``large enough'' behavior occurs once $k+1\geq K$.

Suppose we have a representation of $n$ of length $k+1\ge K+1$ (note the larger bound) that begins with $(1 \ 0 \ T)$, say $n=[1 \ 0 \ T\ a_{k-2} \ a_{k-3} \ \dots \ a_2 \ a_1]_F$. Since $F_{k+1}-F_{k-1}=F_k$, we have that each such representation corresponds bijectively to a representation of $n$ of length $k\ge K$ beginning with $(1\ 0)$, namely, $n=[1 \ 0\ a_{k-2} \ a_{k-3} \ \dots \ a_2 \ a_1]_F$. However, by our assumption on the size of $k$, this tells us $a_{k-2}=T$ as well. A similar argument will show that all representations of $n$ of length $k+1\ge K+1$ that begin with $(T \ 0 \ 1)$ correspond bijectively to all representations of $n$ of length $k$ that begin with $(T \ 0 \ 1)$. Since these are the only representations counted by $f(n;k+1)$ and $f(n;k)$, we see that $f(n;k+1)=f(n;k)$, as desired.

\section{Acknowledgments}
The authors appreciate the support of the Department of Mathematics at the University of Texas at Tyler.  

The authors report there are no competing interests to declare.
\bibliographystyle{plain}
\bibliography{bibliography}

@article{ferns1965representation,
  title={On the representation of integers as sums of distinct {F}ibonacci numbers},
  author={Ferns, H. H.},
  journal={The Fibonacci Quarterly},
  volume={3},
  number={1},
  pages={21--30},
  year={1965},
  publisher={Taylor \& Francis}
}

@article{bicknell1999number,
  title={The number of representations of ${N}$ using distinct {F}ibonacci numbers, counted by recursive formulas},
  author={Bicknell-Johnson, Marjorie and Fielder, Daniel C.},
  journal={The Fibonacci Quarterly},
  volume={37},
  number={1},
  pages={47--60},
  year={1999},
  publisher={Taylor \& Francis}
}

@article{stockmeyer2008smooth,
  title={A smooth tight upper bound for the {F}ibonacci representation function ${R} (N)$},
  author={Stockmeyer, Paul K.},
  journal={The Fibonacci Quarterly},
  volume={46},
  number={2},
  pages={103--106},
  year={2008},
  publisher={Taylor \& Francis}
}

@article{klarner1966representations,
  title={Representations of ${N}$ as a sum of distinct elements from special sequences},
  author={Klarner, David A.},
  journal={The Fibonacci Quarterly},
  volume={4},
  number={4},
  pages={289--306},
  year={1966},
  publisher={Taylor \& Francis}
}

@article{dekking2024counting,
  title={Counting Base Phi Representations},
  author={Dekking, Michel and Loon, Ad van},
  journal={The Fibonacci Quarterly},
  volume={62},
  number={2},
  pages={112--124},
  year={2024},
  publisher={Taylor \& Francis}
}

@article{carlitz1968fibonacci,
  title={Fibonacci representations},
  author={Carlitz, L.},
  journal={The Fibonacci Quarterly},
  volume={6},
  number={4},
  pages={193--220},
  year={1968},
  publisher={Taylor \& Francis}
}

@article{chow2021fibonacci,
  title={On {F}ibonacci partitions},
  author={Chow, Sam and Slattery, Tom},
  journal={Journal of Number Theory},
  volume={225},
  pages={310--326},
  year={2021},
  publisher={Elsevier}
}

@article{alpert2009differences,
  title={Differences of multiple {F}ibonacci numbers.},
  author={Alpert, Hannah},
  journal={Integers},
  volume={9},
  number={6},
  pages={745--749},
  year={2009},
  publisher={Walter de Gruyter}
}

@article{bunder1992zeckendorf,
  title={{Z}eckendorf representations using negative {F}ibonacci numbers},
  author={Bunder, Martin W.},
  journal={The Fibonacci Quarterly},
  volume={30},
  number={2},
  pages={111--115},
  year={1992},
  publisher={Taylor \& Francis}
}

@article{hajnal2023short,
  title={A short note on {Z}eckendorf type numeration systems with negative digits allowed.},
  author={Hajnal, P{\'e}ter},
  journal={Bull. ICA},
  volume={97},
  pages={54--66},
  year={2023}
}

@article{zeckendorf1972representations,
  title={Representations des nombres naturels par une somme de nombres de {F}ibonacci on de nombres de {L}ucas},
  author={Zeckendorf, {\'E}douard},
  journal={Bulletin de La Society Royale des Sciences de Liege},
  pages={179--182},
  year={1972}
}

@article{anders2024non,
  title={Non-Standard Binary Representations and the {S}tern Sequence},
  author={Anders, Katie and Dawsey, Madeline Locus and Gupta, Rajat and Vandehey, Joseph},
  journal={The Electronic Journal of Combinatorics},
  pages={P4--39},
  year={2024}
}

@article{anders2025non,
  title={Non-standard quaternary representations and the {F}ibonacci numbers},
  author={Anders, Katie and Dawsey, Madeline L. and Gupta, Rajat and Lebowitz-Lockard, Noah and Vandehey, Joseph},
  journal={arXiv preprint arXiv:2505.04589},
  year={2025}
}

\end{document}